\documentclass[11pt]{article}
\usepackage{geometry}                
\usepackage{graphicx}
\usepackage{amssymb}
\usepackage{epstopdf}
\usepackage{amsmath,amsthm,amscd,amssymb}
\usepackage{latexsym}
\usepackage[colorlinks,citecolor=red,pagebackref,hypertexnames=false]{hyperref}

\numberwithin{equation}{section}
\theoremstyle{plain}
\newtheorem{theorem}{Theorem}[subsection] 
\newtheorem{lemma}[theorem]{Lemma}
\newtheorem{corollary}[theorem]{Corollary}

\theoremstyle{definition}

\theoremstyle{remark}
\newtheorem{remark}[theorem]{Remark}

\newtheorem{case[theorem]}{Case}

\title{Multi-parameter projection theorems with applications to sums-products and finite point configurations in the Euclidean setting}
\author{B. Erdo\u{g}an, D. Hart and A. Iosevich}
\date{\today}          

\begin{document}
\maketitle

\begin{abstract} In this paper we study multi-parameter projection theorems for fractal sets. With the help of these estimates, we recover results about the size of $A \cdot A+\dots+A \cdot A$, where $A$ is a subset of the real line of a given Hausdorff dimension, $A+A=\{a+a': a,a' \in A \}$ and $A \cdot A=\{a \cdot a': a,a' \in A\}$. We also use projection results and inductive arguments to show that if a Hausdorff dimension of a subset of ${\Bbb R}^d$ is sufficiently large, then the ${k+1 \choose 2}$-dimensional Lebesgue measure of the set of $k$-simplexes determined by this set is positive. The sharpness of these results and connection with number theoretic estimates is also discussed. \end{abstract}    


\section{Introduction} 

\vskip.125in

We start out by briefly reviewing the underpinnings of the sum-product results in the discrete setting. A classical conjecture in geometric combinatorics is that either the sum-set or the product-set of a finite subset of the integers is maximally large. More precisely, let $A \subset {\Bbb Z}$ of size $N$ and define 
$$ A+A=\{a+a': a,a' \in A \}; \ A \cdot A=\{a \cdot a': a,a' \in A \}.$$ 

The Erdos-Szemeredi conjecture says that 
$$ \max \{\#(A+A), \# (A \cdot A) \} \gtrapprox N^2,$$ where here, and throughout, $X \lesssim Y$ means that there exists $C>0$ such that $X \leq CY$ and $X \lessapprox Y$, with the controlling parameter $N$ if for every $\epsilon>0$ there exists $C_{\epsilon}>0$ such that $X \leq C_{\epsilon} N^{\epsilon}Y$. The best currently known result is due to Jozsef Solymosi (\cite{Sol09}) who proved that 
\begin{equation} \label{discreteconjecture}  \max \{\#(A+A), \# (A \cdot A) \} \gtrapprox N^{\frac{4}{3}}. \end{equation} 

For the finite field analogs of these problems, see, for example, \cite{BKT}, \cite{HI08}, \cite{HI08II}, \cite{IR07}, \cite{HIKR10}, \cite{HIS07} and \cite{TV06}. 

The sum product problem has also received considerable amount of attention in the Euclidean setting. The following result was proved by Edgar and Miller (\cite{EM02}) and, independently, by Bourgain (\cite{B03}). 
\begin{theorem} A Borel sub-ring of the reals either has Hausdorff dimension $0$ or is all of the real line. 
\end{theorem} 

Bourgain (\cite{B03}) proved the following quantitative bound that was conjectured in (\cite{KT01}). 
\begin{theorem} Suppose that $A \subset {\Bbb R}$ is a $(\delta, \sigma)$-set in the sense that $A$ is a union of $\delta$-intervals and for $0 < \epsilon \ll 1$,
$$ |A \cap I|<{\left( \frac{r}{\delta} \right)}^{1-\sigma} \delta^{1-\epsilon}$$ whenever $I$ is an arbitrary interval of size $\delta \leq r \leq 1$. Suppose that $0<\sigma<1$ and $|A|>\delta^{\sigma+\epsilon}$. Then 
$$ |A+A|+|A \cdot A|>\delta^{\sigma-c}$$ for an absolute constant $c=c(\sigma)>0$. 
\end{theorem} 

One of the key steps in the proof is the study of the size of 
$$A \cdot A-A \cdot A+A \cdot A-A \cdot A$$ and this brings us to the main results of this paper. Our goal is to show that if the Hausdorff dimension of $A \subset {\Bbb R}$ is sufficiently large, then 
$${\mathcal L}^1(a_1A+a_2A+\dots+a_dA)>0$$ for a generic choice of $(a_1, \dots, a_d) \in A \times A \times \dots \times A$.  It is of note that much work has been done in this direction in the setting of finite fields. See \cite{HIsumprod}, \cite{HIKR10} and the references contained therein. In particular, it is the work in \cite{HIsumprod} that inspired some of the following results.

Our results are proved using generalized projections theorems, similar in flavor to the ones previously obtained by Peres and Schlag (\cite{PS00}) and Solomyak (\cite{Sol98}). 

\begin{theorem} \label{mainprojection} Let $E, F \subset {\Bbb R}^d$, $d \geq 2$, be of Hausdorff dimension $s_E, s_F$, respectively. Suppose that there exist Frostman measures $\mu_E$, $\mu_F$, supported on $E$ and $F$, respectively, such that for $\delta$ sufficiently small and $|\xi|$ sufficiently large, there exist non-negative numbers $\gamma_F$ and $l_F$ such that the following conditions hold: \begin{itemize} 
\item \begin{equation} \label{decay} |\widehat{\mu}_F(\xi)| \lessapprox {|\xi|}^{-\gamma_F} \end{equation} and 
\item \begin{equation} \label{tubesize} \mu_F(T_{\delta}) \lessapprox \delta^{s_F-l_F}, \end{equation} for every tube $T_{\delta}$ of length $\approx 1$ and radius $\approx \delta$ emanating from the origin. 
\end{itemize}

Define, for each $y \in F$, the projection set
$$ \pi_y(E)=\{x \cdot y: x \in E\}.$$ 

Suppose that for some $0< \alpha \leq 1$, 
\begin{equation} \label{maincondition} \max \left\{ \frac{\min\{\gamma_F,s_E\}}{\alpha}, \frac{s_E+s_F-l_F+1-\alpha}{d} \right\}>1.\end{equation} Then 
$$ \dim_{{\mathcal H}}(\pi_y(E)) \geq \alpha$$ for $\mu_F$-every $y \in F$. 

If \eqref{maincondition} holds with $\alpha=0$, then 
$$ {\mathcal L}^1(\pi_y(E))>0$$ for $\mu_F$-every $y \in F$.
\end{theorem} 

\begin{remark} Observe that the conditions of Theorem \ref{mainprojection} always hold with $\gamma_F=0$ and $l_F=1$ since every tube $T_{\delta}$ can be decomposed into $\approx \delta^{-1}$ balls of radius $\delta$. \end{remark} 

\begin{corollary} \label{mainsumproduct} Let $A \subset {\Bbb R}$ and let $\mu_A$ be a Frostman measure on $A$. Then the following hold: \begin{itemize} 
\item Suppose that the Hausdorff dimension of $A$, denoted by $dim_{{\mathcal H}}(A)$, is greater than 
$\frac{1}{2}+\epsilon$ for some $0<\epsilon \leq \frac{1}{2}$. Then 
for $\mu_A \times \mu_A \times \dots \times \mu_A$-every $(a_1, a_2, \dots, a_d) \in A \times A \times \dots \times A$, 
\begin{equation} \label{oralpleasure} \dim_{{\mathcal H}}(a_1A+a_2A+\dots+a_dA) \ge
\min \left\{ 1, \frac{1}{2}+\epsilon(2d-1) \right\}.\end{equation}

\item Suppose that the Hausdorff dimension of $A$ is greater than 
$$\frac{1}{2}+\frac{1}{2(2d-1)}$$ for some $d \geq 2$. Then for $\mu_A \times \mu_A \times \dots \times \mu_A$-every $(a_1, a_2, \dots, a_d) \in A \times A \times \dots \times A$, 
\begin{equation} \label{kickass1} {\mathcal L}^1(a_1A+a_2A+\dots+a_dA)>0.\end{equation}

\item Suppose that $F \subset {\Bbb R}^d$, $d \geq 2$, is star-like in the sense that the intersection of $F$ with every tube of width $\delta$ containing the origin is contained in a fixed number of balls of radius $\delta$. Assume that 
$$ \dim_{{\mathcal H}}(A)+\frac{dim_{{\mathcal H}}(F)}{d}>1. $$ Then for $\mu_F$-every $x \in F$, 
\begin{equation} \label{kickass2} {\mathcal L}^1(x_1A+x_2A+\dots+x_dA)>0.\end{equation} 

In particular, if $dim_{{\mathcal H}}(A)=\frac{1}{2}+\epsilon$, for some $\epsilon>0$, and $dim_{{\mathcal H}}(F)> \frac{d}{2}-\epsilon$, then \eqref{kickass2} holds for $\mu_F$-every $x \in F$. 

\item Suppose that $F \subset {\Bbb R}^d$, $d \geq 2$, possesses a Borel measure $\mu_F$ such that \eqref{decay} holds with $\gamma_F>1$. Suppose that $dim_{{\mathcal H}}(A)>\frac{1}{2}$. Then for 
$\mu_F$-every $x \in F$, 
$$ {\mathcal L}^1(x_1A+x_2A+\dots+x_dA)>0.$$ 
\end{itemize} 
\end{corollary} 

\subsection{Applications to the finite point configurations} 

Recall that the celebrated Falconer distance conjecture (\cite{Falconer86}) says that if the Hausdorff dimension of $E \subset {\Bbb R}^d$, $d \ge 2$, is $>\frac{d}{2}$, then the Lebesgue measure of the set of distances $\{|x-y|: x,y \in E\}$ is positive. The best known result to date, due to Wolff (\cite{Wolff99}) in two dimension and Erdogan (\cite{Erdogan05}) in higher dimensions says that the Lebesgue measure of the distance set is indeed positive if the Hausdorff dimension of $E$ is greater than $\frac{d}{2}+\frac{1}{3}$. 

\begin{corollary} \label{mainspherepinneddistance} Let 
$$E \subset S^{d-1}=\left\{x \in {\Bbb R}^d: \sqrt{x_1^2+x_2^2+\dots+x_d^2}=1 \right\}$$ of Hausdorff dimension $>\frac{d}{2}$. Let $\mu_E$ be a Frostman measure on $E$. Then 
$$ {\mathcal L}^1(\{|x-y|: x \in E\})>0$$ for $\mu$-every $y \in E$. 
\end{corollary} 

Before stating our next result, we need the following definition. Let $T_k(E)$, $1 \leq k \leq d$, denote the $(k+1)$-fold Cartesian product of $E$ with the equivalence relation where $(x^1, \dots, x^{k+1}) \sim (y^1, \dots, y^{k+1})$ if there exists a translation $\tau$ and an orthogonal transformation $O$ such that $y^j=O(x^j)+\tau$. 

In analogy with the Falconer distance conjecture, it is reasonable to ask how large the Hausdorff dimension of $E \subset {\Bbb R}^d$, $d \ge 2$, needs to be to ensure that the ${k+1 \choose 2}$-dimensional Lebesgue measure of $T_k(E)$ is positive. 

\begin{theorem} \label{mainspherekpoint} Let $E \subset S_t=\{x \in {\Bbb R}^d: |x|=t\}$ for some $t>0$. Suppose that $dim_{{\mathcal H}}(E)>\frac{d+k-1}{2}$. Then 
$$ {\mathcal L}^{k+1 \choose 2} (T_k(E))>0.$$ 
\end{theorem} 

Using a pigeon-holing argument, we obtain the following result for finite point configurations in ${\Bbb R}^d$. 

\begin{corollary} \label{mainkpoint} Let $E \subset {\Bbb R}^d$ of Hausdorff dimension $>\frac{d+k+1}{2}$. Then 
$$ {\mathcal L}^{k+1 \choose 2}(T_k(E))>0.$$ 
\end{corollary} 

\begin{remark} We do not know to what extent our results are sharp beyond the following observations. If the Hausdorff dimension of $E$ is less than $\frac{d}{2}$, the classical example due to Falconer (\cite{Falconer86}) shows that the set of distances may have measure $0$, so, in particular, ${\mathcal L}^{k+1 \choose 2}(T_k(E))$ may be $0$ for $k>1$. See also \cite{Falconer85} and \cite{Mat95} for the description of the background material and \cite{Mat85} and \cite{IoSe10} for related counter-example construction.  In two dimensions, one can generalize Falconer's example to show that if the Hausdorff dimension of $E$ is less than $\frac{3}{2}$, then the three dimensional Lebesgue measure of $T_2(E)$ may be zero. In higher dimension, construction of examples of this type is fraught with serious number theoretic difficulties. We hope to address this issue in a systematic way in the sequel. \end{remark} 

\vskip.125in 

\section{Proofs of main results} 

\vskip.125in 

\subsection{Proof of the projection results (Theorem \ref{mainprojection})} 

\vskip.125in 

Define the measure $\nu_y$ on $\pi_y(E)$ by the relation 
$$ \int g(s) d\nu_y(s)=\int g(x \cdot y) d\mu_E(x).$$ 

It follows that 
$$ \int \int {|\widehat{\nu}_y(t)|}^2 dt d\mu_F(y)=\int \int {|\widehat{\mu}_E(ty)|}^2 d\mu_F(y)dt.$$

We have 
$$ \int {|\widehat{\mu}_E(ty)|}^2 d\mu_F(y)=\int \int \widehat{\mu}_F(t(u-v)) d\mu_E(u) d\mu_E(v)$$
$$=\int \int_{|u-v| \leq t^{-1}} \widehat{\mu}_F(t(u-v)) d\mu_E(u) d\mu_E(v)$$
$$+\int \int_{|u-v|>t^{-1}} \widehat{\mu}_F(t(u-v)) d\mu_E(u) d\mu_E(v)=I+II.$$ 

Since $\mu_E$ is a Frostman measure, 
$$ |I| \lessapprox t^{-s_E}.$$ 

By \eqref{decay}, 
$$ |II| \lessapprox t^{-\gamma_F} \int \int {|u-v|}^{-\gamma_F} d\mu(u)d\mu(v) \lesssim t^{-\gamma_F}$$ as long as $\gamma_F \leq s_E$. If $\gamma_F>s_E$, then for any $\epsilon>0$, 
$$ |II| \lesssim t^{-s_E+\epsilon} \int \int {|u-v|}^{-s_E+\epsilon} d\mu(u)d\mu(v) \lessapprox t^{-s_E},$$ and we conclude that 
\begin{equation} \label{decayuptheass}  \int {|\widehat{\mu}_E(ty)|}^2 d\mu_F(y) \lessapprox t^{-\min \{s_E, \gamma_F\}}.\end{equation} 

It follows that 
$$ \int \int {|\widehat{\nu}_y(t)|}^2 t^{-1+\alpha} d\mu_F(y)dt<\infty$$ if $\min \{s_E, \gamma_F\}>\alpha$. 

\vskip.125in 

We now argue via the uncertainty principle. We may assume, by scaling and pigeon-holing, that $F \subset \{x \in {\Bbb R}^d: 1 \leq |x| \leq 2\}$. Let $\phi$ be a smooth cut-off function supported in the ball of radius $3$ and identically equal to $1$ in the ball of radius $2$. It follows that 
$$ \int {|\widehat{\nu}_y(t)|}^2 t^{-1+\alpha} dt$$
$$ =\int \int {|\widehat{\mu}_E(ty)|}^2 \ t^{-1+\alpha} \ dt \ d\mu_F(y)=
\int \int {\left|\widehat{\mu}_E*\widehat{\phi}(ty)\right|}^2 d\mu_F(y) \ t^{-1+\alpha} \ dt$$
$$ \lesssim \int \int \int {|\widehat{\mu}_E(\xi)|}^2 |\widehat{\phi}(ty-\xi)| d\mu_F(y) \ t^{-1+\alpha} \ dt d\xi$$
$$ \lesssim \sum_m 2^{-mn} 
\int {|\widehat{\mu}_E(\xi)|}^2 \mu_F \times {\mathcal L}^1 \{(y,t): |ty-\xi| \leq 2^m\} {|\xi|}^{-1+\alpha} d\xi$$
$$ \lesssim \sum_m 2^{-mn} \cdot 2^m 
\int {|\widehat{\mu}_E(\xi)|}^2 \mu_F \left( T_{{2^{m}|\xi|}^{-1}}(\xi) \right) {|\xi|}^{-1+\alpha} d\xi,$$ where 
$T_{\delta}(\xi)$ is the tube of width $\delta$ and length 10 emanating from the origin in the direction of 
$\frac{\xi}{|\xi|}$. By assumption, this expression is 
$$ \lessapprox \sum_m 2^{-mn} \cdot 2^m \cdot 2^{m(s_F-l_f)} 
\int {|\widehat{\mu}_E(\xi)|}^2 {|\xi|}^{-s_F+l_F-1+\alpha} d\xi \lesssim 1$$ if 
$$ s_F-l_F+1-\alpha>d-s_E$$ and $n$ is chosen to be sufficiently large. Combining this with \eqref{decayuptheass} we obtain the conclusion of Theorem \ref{mainprojection}. 

\vskip.125in 

\subsection{Proof of applications to sums and products (Corollary \ref{mainsumproduct})} 

\vskip.125in 
Let $A\subset {\Bbb R}$ have dimension greater than $s_A:=\frac{1}{2}+\frac{1}{2(2d-1)}$.
Note that we can find a probability measure, $\mu_A$, supported on $A$ satisfying
\begin{align}\label{eq:muA1}
\mu_A(B(x,r))&\leq Cr^{s_A}, \quad x\in {\Bbb R}, r>0, \end{align}

Let $E=A \times A \times \dots \times A$. Define $\mu_E=\mu_A\times \mu_A\times \cdots \times \mu_A$. 

\begin{lemma} \label{tube} With the notation above, 
\begin{equation} \label{keysumproduct} \mu_E \left( T_{\delta} \right) \lessapprox {\delta}^{(d-1)s_A},\end{equation} where $dim_{{\mathcal H}}(A)=s_A$. 
\end{lemma} 
\begin{proof}
Let $l_{\xi}=\{s \xi: s \in {\Bbb R}\}$ and assume without loss of generality that $\xi_1$ is the largest coordinate of $\xi$ in absolute value. In particular, $\xi_1 \not=0$. Define the function 
$$ \Psi: A \rightarrow (A \times A \times \dots \times A) \cap l_{\xi}$$ by the relation 
$$ \Psi(a)=\left(a, a \frac{\xi_2}{\xi_1}, \dots, a \frac{\xi_d}{\xi_1} \right).$$ 
Note that 
\begin{align*}
\mu\left( T_{\delta}(\xi) \right)&=\mu_A\times\cdots\times \mu_A\left( T_{\delta}(\xi) \right)\\
&\leq \int_{-10}^{10} \mu_A\times\cdots\times \mu_A(B(\Psi(x_1), \delta))d\mu_A(x_1)\\
&\lesssim \int_{-10}^{10} \delta^{(d-1) s_A}d\mu_A(x_1)\lesssim \delta^{(d-1) s_A}.
\end{align*}
\end{proof}

It follows that if $E=F=A \times A \times \dots \times A$, then the assumptions of Theorem \ref{mainprojection} are satisfied with $s_E=s_F=d s_A$, $\gamma_F=0$ and $l_F=\frac{s_E}{d}$. The conclusion of the first part of Corollary \ref{mainsumproduct} follows in view of Theorem \ref{mainprojection}. 

The second conclusion of Corollary \ref{mainsumproduct} follows, in view of Theorem \ref{mainprojection}, if we observe that if $F$ is star-like, then 
$$ \mu_F(T_{\delta}) \lessapprox \delta^{-s_F}.$$ 

The third conclusion of Corollary \ref{mainsumproduct} follows from Theorem \ref{mainprojection} since we may always take $l_F=1$. This holds since every tube $\delta$ is contained in a union of 
$\delta^{-1}$ balls of radius $\delta$. 

\vskip.125in 

\subsection{Proof of the spherical configuration result (Theorem \ref{mainspherekpoint})} 

\vskip.125in 

Let $y=(y^1,y^2, \dots, y^{k})$, $y^j \in E$ and define 
$$ \pi_y(E)=\{x \cdot y^1, \dots, x \cdot y^{k}:x\in E\}.$$ 

Define a measure on $\pi_y(E)$ by the relation 
$$ \int g(s) d\nu_y(s)=\int g(x \cdot y^1, \dots, x \cdot y^{k}) d\mu(x),$$ where 
$s=(s_1, \dots, s_{k})$ and $\mu$ is a Frostman measure on $E$. It follows that 
$$ \widehat{\nu}_y(t)=\widehat{\mu}(t \cdot y),$$ where $t=(t_1, \dots, t_{k})$ and 
$$ t \cdot y=t_1y^1+t_2y^2+\dots+t_{k}y^{k}.$$ It follows that 
$$ \int \int {|\widehat{\nu}_y(t)|}^2 dt d\mu^{*}(y)=\int \int {|\widehat{\mu}(t \cdot y)|}^2 dt d\mu^{*}(y),$$ where 
$$ d\mu^{*}(y)=d\mu(y^1) d\mu(y^2) \dots d\mu(y^{k}).$$ 

Arguing as above, this quantity is bounded by 
$$ \int \int \int {|\widehat{\mu}(\xi)|}^2 |\widehat{\phi}(t \cdot y-\xi)| dt d\mu^{*}  d\xi.$$ 

It is not difficult to see that 
$$ \int \int  |\widehat{\phi}(t \cdot y-\xi)| dt d\mu^{*} \lesssim {|\xi|}^{-s+k-1}$$ since once we fix a linearly independent collection $y^1, y^2, \dots, y^{k-1}$, $y^{k}$ is contained in the intersection of $E$ with a $k$-dimensional plane. Since $E$ is also a subset of a sphere, the claim follows. If $y^1, \dots, y^{k-1}$ are not linearly independent, the estimate still holds, but the easiest way to proceed is to observe that since the Hausdorff dimension of $E$ is greater than $k$ by assumption, there exist $k-1$ disjoint subsets $E_1, E_2, \dots, E_{k-1}$ of $E$ and a constant $c>0$ such that $\mu(E_j) \ge c$ and any collection $y^1, \dots, y^{k-1}$, $y^j \in E_j$, is linearly independent. This establishes our claim with $\mu^{*}$ replaced by the product measure restricted to $E_j$s, which results in the same conclusion. 

Plugging this in we get 
$$ \int {|\widehat{\mu}(\xi)|}^2 {|\xi|}^{-s+k-1} d\xi<\infty$$ if 
$$ s>\frac{d}{2}+\frac{k-1}{2}.$$ This implies that for $\mu^k$ almost every $k$-tuple $y=(y^1,y^2,\ldots,y^{k})\in E^k$, $\nu_y$ is absolutely continuous w.r.t $\mathcal{L}^k$ and hence its support, $\pi_y(E)$, is of positive $\mathcal{L}^k$ measure.

Let 
$$\mathcal{E}_k=\{y=(y^1,\ldots,y^k)\in E^k:\mathcal{L}^k(\pi_y(E))>0\}.$$
We just proved that $\mu^k(\mathcal{E}_k)=\mu^k(E^k)>0$.

Consider the set 
$$P_{k-1}=\big\{(y^1,\ldots,y^{k-1})\in E^{k-1}: \mu\big(\{x:(y^1,\ldots,y^{k-1},x)\in\mathcal{E}_k\}\big)=\mu(E)\big\}.$$
By the discussion above, and Fubini,  $\mu^{k-1}(P_{k-1})=\mu(E)^{k-1}>0$.

For each $y=(y^1,\ldots,y^{k-1})\in P_{k-1}$, let 
$F_y=\{x\in E:(y^1,\ldots,y^{k-1},x)\in \mathcal{E}_k\}$, and define
$$\pi_{y}(F_y)=\{(x\cdot y^1,\ldots,x\cdot y^{k-1}):x\in F_y\}.$$
As above we construct the measure $\nu_y$ supported on $\pi_{y}(F_y)$, and we have 
$$\widehat{\nu_y}(t)=\widehat{\mu\chi_{F_y}}(t\cdot y)=\widehat{\mu}(t\cdot y),$$
since $\mu(E\backslash F_y)=0$. By the argument above, we conclude that $\pi_y(F_y)$ is of positive 
$\mathcal{L}^{k-1}$ measure. Therefore, by Fubini, for $\mu^{k-1}$ a.e. $(y^1,\ldots,y^{k-1})\in E^{k-1}$,
$$\mathcal{L}^{k+k-1}\{y^1\cdot y,\ldots,y^{k-1}\cdot y,y^1\cdot x,\ldots,y^{k-1}\cdot x, y\cdot x:
x,y\in E \}>0.$$
The result now follows from this induction step.

\vskip.25in 

\section{Proof of the Euclidean configuration result (\ref{mainkpoint})} 

\vskip.125in 

We shall make use of the following intersection result. See Theorem 13.11 in \cite{Mat95}. 

\begin{theorem} Let $a,b>0, a+b>d$, and $b>\frac{d+1}{2}$. If $A,B$ are Borel subsets in ${\Bbb R}^d$ with ${\mathcal H}^a(A)>0$ and ${\mathcal H}^b(B)>0$, then for almost all $g \in O(d)$, 
$$ {\mathcal L}^d \{z \in {\Bbb R}^d: dim_{{\mathcal H}}(A \cap (\tau_z \circ g)B) \ge a+b-d \}>0,$$ where $\tau_z$ denotes the translation by $z$. 
\end{theorem} 

In the special case when $B$ is the unit sphere in dimension $4$ or higher, we get the following corollary. 
\begin{corollary} Let $E \subset {\Bbb R}^d$, $d \ge 4$, of Hausdorff dimension $s>1$. Then 
$$ {\mathcal L}^d \{z \in {\Bbb R}^d: dim_{{\mathcal H}}(E \cap (S^{d-1}+z)) \ge s-1 \}>0.$$ 
\end{corollary} 

It follows from the corollary that if $E \subset {\Bbb R}^d$, $d \ge 4$, is of Hausdorff dimension $>\frac{d+k+1}{2}$, there exists $z \in {\Bbb R}^d$, such that the Hausdorff dimension of $E \cap (z+S^{d-1})$ is $>\frac{d+k-1}{2}$. Now observe that if $x, y \in z+S^{d-1}$, means that $x=x'+z$, $y=y'+z$, where $x',y' \in S^{d-1}$. It follows that 
$$ {|x-y|}^2={|x'-y'|}^2=2-2x' \cdot y'.$$ 

In other words, the problem of simplexes determined by elements of $E \cap (z+S^{d-1})$ reduces to Theorem \ref{mainspherekpoint} and thus Theorem \ref{mainkpoint} is proved.

\end{document}